\documentclass{amsart}
\usepackage{amssymb,amsmath,amsthm,xspace}
\usepackage[all]{xy}
\usepackage{graphics,graphicx}
\graphicspath{{images/}}

\def\mat#1{\ensuremath{#1}\xspace}

\def\cN{\mat{\mathbb{N}}}   
\def\cQ{\mat{\mathbb{Q}}}   
\def\cR{\mat{\mathbb{R}}}   
\def\cC{\mat{\mathbb{C}}}   

\def\cZ{\mat{\mathbb{Z}}}   

\def\lE{\mat{\mathcal{E}}}

\def\lH{\mat{\mathcal{H}}}

\def\lP{\mat{\mathcal{P}}}

\def\al{\mat{\alpha}}
\def\be{\mat{\beta}}
\def\ga{\mat{\gamma}}

\def\de{\mat{\delta}}
\def\De{\mat{\Delta}}

\let\etatemp\eta \def\eta{\mat{\etatemp}}
\def\hi{\mat{\chi}}

\let\xitemp\xi \def\xi{\mat{\xitemp}}
\def\La{\mat{\Lambda}}
\def\la{\mat{\lambda}}
\let\mutemp\mu\def\mu{\mat{\mutemp}}

\let\nutemp\nu\def\nu{\mat{\nutemp}}
\def\Om{\mat{\Omega}}
\def\om{\mat{\omega}}
\let\Phitemp\Phi \def\Phi{\mat{\Phitemp}}

\let\Psitemp\Psi \def\Psi{\mat{\Psitemp}}
\let\pitemp\pi\def\pi{\mat{\pitemp}}
\let\rhotemp\rho\def\rho{\mat{\rhotemp}}

\def\te{\mat{\theta}}

\def\mrm@#1{\mat{\mathrm{#1}}}

\def\DMO{\DeclareMathOperator}

\DMO{\Hom}{Hom}
\DMO{\lHom}{\lH\mathit{om}}
\DMO{\Ext}{Ext}
\DMO{\lExt}{\lE\mathit{xt}}
\DMO{\End}{End}
\DMO{\Aut}{Aut}
\DMO{\Fun}{Fun}
\DMO{\Tor}{Tor}
\DMO{\ext}{ext}

\DMO{\Ob}{Ob}
\DMO{\Mor}{Mor}
\DMO{\im}{im}
\DMO{\coim}{coim}
\DMO{\coker}{coker}
\DMO{\Arr}{Arr}
\DMO{\Id}{Id}
\DMO{\id}{id}
\DMO{\add}{add} 
\DMO{\ind}{ind} 
\DMO{\pro}{pro} 
\DMO{\Map}{Map} %
\DMO{\Iso}{Iso} %
\DMO{\Isom}{Isom}%
\DMO{\Ind}{Ind}

\DMO{\Presh}{Presh}

\DMO\coalg{Coalg}
\DMO{\Rep}{Rep}
\DMO{\Cor}{Cor}

\DMO{\Mod}{Mod}
\DMO{\rad}{rad}
\DMO{\soc}{soc}
\DMO{\ann}{ann}

\DMO{\Spec}{Spec}
\DMO{\spec}{Spec}
\DMO{\Proj}{Proj}
\DMO{\supp}{supp}
\DMO{\Coh}{Coh}
\DMO{\coh}{Coh}
\DMO{\Qcoh}{QCoh}
\DMO{\QCoh}{QCoh}
\DMO{\Pic}{Pic}
\DMO{\Div}{Div}
\DMO{\ch}{ch}
\DMO{\Hilb}{Hilb}
\DMO{\Fitt}{Fitt}
\DMO{\Quot}{Quot}

\DMO{\Gras}{Gr}
\DMO{\Flag}{Flag}

\DMO{\cone}{cone}
\DMO{\Tw}{Tw}
\DMO{\rank}{rk}
\DMO{\rk}{rk}
\DMO{\codim}{codim}
\DMO{\cov}{cov}
\DMO{\sgn}{sgn}
\DMO{\td}{td}
\DMO{\GL}{GL}
\DMO{\SL}{SL}

\DMO\Der{Der}
\DMO\der{Der}
\DMO\coder{Coder}
\DMO{\diag}{diag}
\DMO{\HMod}{HMod} 
\DMO{\ad}{ad}
\DMO{\Ad}{Ad}
\DMO*{\colim}{colim}
\DMO*{\hocolim}{hocolim}
\DMO*{\holim}{holim}
\DMO{\Ho}{Ho}

\DMO{\har}{char}
\DMO{\sk}{sk}
\DMO{\cosk}{cosk}
\DMO{\Gal}{Gal}
\DMO{\tr}{tr}
\DMO{\Tr}{Tr}
\DMO{\Sh}{Sh}
\DMO{\Is}{Is} 
\DMO{\Hol}{Hol} 
\DMO{\Lie}{Lie} 
\DMO{\Res}{Res} 
\DMO{\irr}{irr} %
\DMO{\Irr}{Irr} %
\DMO{\Exp}{Exp} %
\DMO{\Log}{Log} %
\DMO{\Pow}{Pow}
\DMO{\pow}{pow}
\DMO{\mult}{mult} %
\DMO{\height}{ht} %
\DMO{\wt}{wt}
\DMO{\Vect}{Vect}
\DMO{\moda}{mod}
\DMO{\hd}{hd} 

\def\dd{\mat{\partial}}
\def\iso{\simeq}
\def\ts{\otimes}
\def\tl#1{\mat{\tilde{#1}}}
\def\wtl#1{\mat{\widetilde{#1}}}
\def\what#1{\mat{\widehat{#1}}}

\def\sb{\subset}

\def\sp{\supset}
\def\imp{\mat{\Rightarrow}}

\def\xx{\times}

\def\ms{\backslash} 
\def\pser#1{[\![#1]\!]} 
\DMO{\Eig}{Eig} 

\def\inv{^{-1}}
\def\dual{^\vee}

\def\sst{^{ss}}

\def\ang#1{\mat{\left\langle #1\right\rangle}}
\def\angs#1#2{\mat{\left\langle #1 \mid #2\right\rangle}}
\def\set#1{\mat{\{ #1\}}}
\def\sets#1#2{\mat{\{ #1 \mid #2\}}}

\def\emb{\hookrightarrow}
\def\mto{\mapsto}
\def\arr{\futurelet\test\arrtest}
\def\arrtest{\ifx^\test\let\next\arra\else\let\next\arrb\fi\next}
\def\arra^#1{\xrightarrow{#1}} \def\arrb{\to}

\def\arrowsD{
\def\mto{{\:\vrule height .9ex depth -.2ex width .04em\!\!\!\;\ar}}
\def\ar{{\:\vrule depth -.52ex height .60ex width 0.85em\;\!\!\rhla\,}}
\def\arr{\futurelet\test\arrtest}
\def\arrtest{\ifx^\test\let\next\arra\else\let\next\arrb\fi\next}
\def\arra^##1{\rTo^{##1}} \def\arrb{\ar}
\def\emb{\futurelet\test\embtest}
\def\embtest{\ifx^\test\let\next\emba\else\let\next\embb\fi\next}
\def\emba^##1{\rInto^{##1}} \def\embb{{\:\rthooka\!\!\!\ar}}
\newarrow{Eq}=====
\def\rrarr{\pile{\rTo\\ \rTo}}
\def\lrarr{\pile{\rTo\\ \lTo}}   
\newarrow{ShortTo}{}{}-->
}

\def\arrowsDStandard{
\newarrow{TeXto}----{->}
\newarrow{TeXinto}C---{->}
\newarrow{TeXonto}----{->>}
\newarrow{TeXdashto}{}{dash}{}{dash}{->}
\newarrow{Eq}=====
\def\ar{\rightarrow}
\def\emb{\futurelet\test\embtest}
\def\embtest{\ifx^\test\let\next\emba\else\let\next\embb\fi\next}
\def\emba^##1{\rTeXinto^{##1}} \def\embb{\hookrightarrow}
}

\newif\ifukr\ukrfalse
\newif\ifrus\rusfalse
\newif\ifger\gerfalse

\def\theorems{
\newcounter{nthr} 
\numberwithin{nthr}{section}
\newtheorem{thr}[nthr]{Theorem}
\newtheorem{prp}[nthr]{Proposition}
\newtheorem{lmm}[nthr]{Lemma}
\newtheorem{crl}[nthr]{Corollary}
\newtheorem{clm}[nthr]{Claim}
\newtheorem{conj}[nthr]{Conjecture}
\theoremstyle{definition}
\newtheorem{dfn}[nthr]{Definition}
\newtheorem{rmr}[nthr]{Remark}
\newtheorem{exm}[nthr]{Example}
\newtheorem{claim}[nthr]{Claim}}

\def\ub#1{\mat{\overline{#1}}}  

\theorems
\newtheorem{condition}[nthr]{Condition}
\usepackage{hyperref}

\def\q{\mat{Q}}	
\def\pq{\mat{\wtl Q}} 
\def\bg{\mat{G}} 
\def\pc{\mat{\lP\q}} 
\def\pg{\mat{\lP_{\rm w}\q}} 
\def\fc{\mat{\pc/(\dd W)}} 
\def\fg{\mat{\pg/(\dd W)}} 
\def\hat{\what}

\begin{document}
\title[Noncommutative Donaldson-Thomas invariants]
{On the noncommutative Donaldson-Thomas invariants arising from brane tilings}%

\author{Sergey Mozgovoy}%
\author{Markus Reineke}%


\email{mozgov@math.uni-wuppertal.de}%
\email{reineke@math.uni-wuppertal.de}%


\begin{abstract}
Given a brane tiling, that is a bipartite graph on a torus,
we can associate with it a quiver potential and a quiver
potential algebra.
Under certain consistency conditions on a brane tiling, 
we prove a formula for the Donaldson-Thomas type invariants
of the moduli space of framed cyclic modules over the
corresponding quiver potential algebra. We relate this
formula with the counting of perfect matchings of the periodic
plane tiling corresponding to the brane tiling.
We prove that the same consistency conditions
imply that the quiver potential algebra is a 3-Calabi-Yau
algebra. We also formulate a rationality conjecture for the
generating functions of the Donaldson-Thomas type invariants.
\end{abstract}
\maketitle
\tableofcontents
\section{Introduction}
The main objective of this paper is to generalize the results
of Szendr\H oi \cite{Szen1} on the noncommutative Donaldson-Thomas theory
in the case of the conifold to the case of quiver potentials arising from
arbitrary brane tilings (see Section \ref{sec:brane}).
That is, we compute the Donaldson-Thomas type invariants 
\cite{BehrFant1} of the moduli spaces of framed cyclic modules
over the quiver potential algebra. 

The Donaldson-Thomas type invariants
are the weighted Euler numbers and they can therefore be computed using
the localization technique, whenever we can find an appropriate torus action
on the moduli space. In this paper we show that the 
quiver potential algebra arising from the brane tiling has a canonical grading.
This grading provides us with a torus action on the moduli
space, so that there is a finite number of fixed points, and we
get a purely combinatorial formula that counts ideals in some
poset of paths in the quiver (see Proposition \ref{prp:formula1}).

This nice state of affairs will not be for free. We need
to impose certain conditions on the brane tiling 
(see conditions \ref{cond:B}, \ref{cond:C}),
which we call the consistency conditions.
These conditions should be compared with the consistency
conditions arising in physics \cite{Hana1}.
It follows from the results of \cite{HHV, Broomhead1}
that our consistency conditions are satisfied under the physical 
consistency conditions (see also Remark \ref{rmr:R-charge}).

In the presence of the consistency conditions, we 
construct a bijection between the set of ideals in the poset of paths 
mentioned above and the set of perfect matchings
(having certain prescribed behavior at infinity)
of the periodic plane tiling induced by the brane tiling.
This bijection generalizes the folklore result on the correspondence
between the $3$-dimensional Young diagrams and the so-called honeycomb
dimers (see Section \ref{sec:perf mat}).
Earlier Szendr\H oi \cite{Szen1}
observed such correspondence in the conifold case, where it also
has a visual interpretation like for the $3$-dimensional Young diagrams.

Furthermore, we prove that, under the consistency conditions,
the quiver potential algebra is always a $3$-Calabi-Yau algebra
(and actually a graded $3$-Calabi-Yau algebra, when provided with a
canonical grading).
This does not contradict the results of Bocklandt \cite[Theorem 3.1]{Bock1},
who proved that a graded quiver potential algebra can be a graded $3$-Calabi-Yau
algebra only under some rather restrictive conditions (for example,
the algebra $\cC[x,y,z]$ does not satisfy those conditions) because
the grading group in \cite{Bock1} can only be \cZ and all arrows have degree one there. 
We learned that Nathan Broomhead \cite{Broomhead1} recently proved the $3$-Calabi-Yau 
property under the condition that there exists an $R$-charge on
the brane tiling (see Remark \ref{rmr:R-charge}).

Our formulas give a way to compute the non-commutative DT
invariants with a computer. In the cases of the orbifolds
$\cC^3/\cZ_n$, $\cC^3/(\cZ_2\xx\cZ_2)$ and in the case of the conifold
there exist nice compact formulas due to Benjamin Young \cite{Young1}.
Using the results of Young and computer evidence in other cases,
we formulate a rationality conjecture on the non-commutative DT
invariants in Section \ref{sec:conj}.
We learned from Nagao Kentaro about joint work with Hiraku
Nakajima \cite{Nagao1} \cite{Nagao2} on the proof of Young's formulas in the conifold case
using wall-crossing formulas. 
Their technique could possibly provide
a formula for general brane tilings.

The paper is organized as follows: 
In Section \ref{sec:hilbert_schemes} we construct the moduli
spaces of framed cyclic modules and describe the general localization
technique to reduce the problem of computation of the Euler number
of the moduli space to some combinatorial problem.
In Section \ref{sec:brane} we consider the quiver potentials
associated with brane tilings and study the canonical grading
of the quiver potential algebra. In Section \ref{sec:groupoid}
we study different equivalence relations on paths
induced by a potential. In Section \ref{sec:perf mat} we prove
a bijection between the set of finite ideals in the poset of paths
and the set of perfect matchings of the periodic plane tiling having
some prescribed behavior at infinity. In Section \ref{sec:CY} we prove that
the quiver potential algebra associated to a brane tiling satisfying
the consistency conditions is a $3$-Calabi-Yau algebra. In Section
\ref{sec:dt invariants}
we relate the Donaldson-Thomas type invariants of moduli spaces
of framed cyclic modules with their Euler numbers.
In Section \ref{sec:conj} we formulate the rationality conjecture.

We would like to thank Raf Bocklandt and Alastair King for many
helpful discussions and useful comments. We would like to thank
Tom Bridgeland for pointing out a mistake in Section \ref{sec:conj}
in the first version of the paper.
\section{Hilbert schemes}\label{sec:hilbert_schemes}
Let \q be a quiver, let $I\sb\cC\q$ be an ideal
of the path algebra, and let $A=\cC\q/I$ be the factor algebra.
For any $i\in\q_0$ there is an idempotent $e_i\in A$.
The $A$-module $P_i:=Ae_i$ is a projective $A$-module.
For any $A$-module $M$, we define a vector space $M_i:=e_iM$.
Then $M=\oplus_{i\in\q_0}M_i$. 
We define $\dim M:=(\dim M_i)_{i\in\q_0}\in\cN^{\q_0}$.

\begin{dfn}
For any $i\in\cQ_0$, we define an $i$-cyclic $A$-module
to be a pair $(M,m)$ where $M$ is a finite-dimensional
left $A$-module, $m\in M_i$, and $Am=M$. 
The $i$-cyclic $A$-modules form a category, where morphisms
$f:(M,m)\arr(N,n)$ are $A$-module homomorphisms
$f:M\arr N$ such that $f(m)=n$.
\end{dfn}

It is clear that the $i$-cyclic $A$-modules correspond
to the finite-dimensional quotients of $P_i$. 
We want to construct the moduli
spaces of isomorphism classes of $i$-cyclic $A$-modules, 
which we will call the Hilbert schemes.

Let $\hat\q$ be a new quiver with 
$$\what\q_0=\q_0\cup\set*,\qquad \hat\q_1=\q_1\cup\set{a_*:*\arr i}.$$
Let $\hat I\sb\cC\hat\q$
be the ideal generated by $I$. Then a module over 
$\hat A=\cC\hat\q/\hat I$ can be identified with a triple $(M,V,f)$,
where $M$ is an $A$-module, $V$ is a vector space, and 
$f:V\arr M_{i}$ is a linear map.

\begin{dfn}
Given an element $\te\in\cN^{\q_0}$, called a stability,
we define a slope function
$$\mu:\cN^{Q_0}\ms\set0\arr\cQ,\qquad \al\mto\frac{\te\cdot\al}{\sum_{i\in\q_0}\al_i}.$$
For any finite-dimensional nonzero $A$-module $M$, we define
$\mu(M):=\mu(\dim M)$. The module $M$ is called
\te-semistable (resp. \te-stable) if for any its proper nonzero
submodule $N\sb M$ we have $\mu(N)\le\mu(M)$
(resp. $\mu(N)<\mu(M)$). The stability condition
for $\hat A$-modules is defined in the same way.
\end{dfn}

Given a dimension vector $\al\in\cN^{\q_0}$, we define a new dimension vector
$\hat\al=(\al,1)\in\cN^{\q_0}\xx\cN=\cN^{\hat\q_0}$.
Define a stability $\hat\te=(0,\dots,0,1)\in\cN^{\q_0}\xx\cN=\cN^{\hat\q_0}$.

\begin{lmm}
Let $(M,V,f)$ be an $\hat A$-module  of dimension $\hat\al$.
Then the following conditions are equivalent
\begin{enumerate}
	\item $(M,V,f)$ is $\hat\te$-semi-stable.
	\item $(M,V,f)$ is $\hat\te$-stable.
	\item  $f(V)$ generates $M$ as an $A$-module.
\end{enumerate}
\end{lmm} 

This lemma implies that the moduli space
$$\Hilb^\al_{i}(A):=M\sst_\te(\hat A,\hat\al)$$
of semistable $\hat A$-modules \cite{King1}
parametrizes the $i$-cyclic modules, or equivalently,
the quotients of $P_{i}$ of dimension~$\al$.

The goal of this paper is to study the Donaldson-Thomas type
invariant of $\Hilb^\al_{i}(A)$. It is defined as a certain 
weighted Euler number (see e.g. \cite{BehrFant1}). 
We will first develop
certain techniques to compute the usual Euler number of
this Hilbert scheme and then prove in Section \ref{sec:dt invariants}
that its DT-invariant differs from the Euler number
just by sign (in the case of quiver potentials induced by brane tilings). 
This should be compared with the dimension
zero MNOP conjecture \cite[Theorem 4.12]{BehrFant1}, \cite{MNOP1}.
We define a partition function
$$Z^i(A)=\sum_{\al\in\cN^{\q_0}}\hi_c(\Hilb^\al_{i}(A))x^\al\in\cQ\pser{x_j|j\in\q_0},$$
where $\hi_c$ means the Euler number of cohomology
with compact support.

Let $\wt:\q_1\arr\La$ be a map, called a weight function, to a free abelian
finitely generated group \La. The path algebra $\cC\q$ is then 
automatically a \La-graded algebra. 
We assume that $I\sb\cC\q$ is a \La-homogeneous ideal.
The quotient algebra $A=\cC\q/I$ is again a
\La-graded algebra.

We define the action of the torus $T=\Hom_\cZ(\La,\cC^*)$
on the Hilbert scheme $\Hilb^\al_{i}(A)$
of $i$-cyclic $A$-modules as follows. 
For any $t\in T$ and any $i$-cyclic $A$-module $(M,m)$, 
we define $(M',m')=t(M,m)$ 
by $M'_i=M_i$ for $i\in\q_0$, $m'=m$ and
$$M'_a=t_aM_a,\qquad\text{for }a\in\q_1,$$
where $t_a:=t(\wt(a))$.
By localization, we have
$$\hi_c(\Hilb^\al_{i}(A))=\hi_c(\Hilb^\al_{i}(A)^T),$$
where $\Hilb^\al_{i}(A)^T$ is the subvariety of $T$-fixed points.

\begin{thr}
An $i$-cyclic $A$-module $(M,m)$ in $\Hilb^\al_{i}(A)$
is $T$-fixed if and only if $M$ possesses a $\La$-grading
as an $A$-module such that $m$ has degree zero 
(such a grading is unique as $m$ generates $M$).
\end{thr}
\begin{proof}
Let $(M,m)$ be some $T$-fixed point of $\Hilb^\al_{i_0}(A)$, $i_0\in\q_0$. 
For any $t\in T$ there exists $g=(g_i)_{i\in\q_0}\in\GL_\al$
such that for any arrow $a:i\arr j$, we have (recall that $t_a=t(\wt(a))$)
$$t_aM_a=g_jM_ag_i\inv$$
and $g_{i_0}m=m$. 
Consider the subgroup $H\sb T\xx\GL_\al$ of all pairs
$(t,g)$ that satisfy this condition. Then $p_1:H\arr T$ is 
surjective. But its kernel is trivial. It consists of
pairs $(1,g)$, where $g$ is an automorphism of $M$
that fixes $m$. It follows that $g$ acts trivially on $M$,
as $m$ generates $M$.
Consider the composition 
$$\psi=p_2\circ p_1\inv:T\arr\GL_\al$$
and split it to components $\psi_i:T\arr\GL(M_i)$, $i\in\q_0$.
Then for any arrow $a:i\arr j$, we have 
$$t_aM_a=\psi_j(t_a)M_a\psi_i(t_a)\inv$$
and $\psi_{i_0}(t)m=m$ for any $t\in T$.
Using the action of $T$ on $M_i$ defined by $\psi_i$, we can decompose $M_i$
with respect to the character group $X(T)\iso\La$ of $T$
$$M_i=\bigoplus_{\la\in\La}M_{i,\la}.$$
Then the above condition implies that for any arrow $a:i\arr j$
$$M_a(M_{i,\la})\sb M_{j,\la+\wt(a)}.$$
and $m\in M_{i_0,0}$. This means that $M$ is a \La-graded
$A$-module and $m$ has degree zero.
The converse statement is easy.
\end{proof}

An $i$-cyclic $A$-module $(M,m)$ as in the above theorem, 
will be called a \La-graded $i$-cyclic $A$-module.
Given a path $u$ in \q, let us define its weight $\wt(u)\in\La$
to be the sum of the weights of arrows from $u$. We assume that
\begin{enumerate}
	\item Any two paths in \q having the same startpoint
	and the same weight are proportional in $A$.
	\item The weight of any non-trivial path is non-zero.
\end{enumerate}

The first assumption implies that the \La-graded module $P_{i}$ 
has dimension at most one at every degree. 
It follows that any \La-graded quotient of $P_i$ is determined
by the set of weights from its support.
We define $\De=\De_i$ to be the set
of paths starting at $i$ modulo an equivalence relation
$u\sim v$ if $\wt(u)=\wt(v)$. 
It follows from our assumptions that there is a poset structure 
on \De, given by the rule $u\le v$ 
if there exists some path $w$ with $wu\sim v$.
There is a bijection between \De and the set of weights $\la\in\La$
such that $P_{i,\la}\ne0$.

\begin{lmm}\label{lmm:La-graded mod. and ideals}
There is a bijection between the set of isomorphism classes
of \La-graded $i$-cyclic $A$-modules and the set of finite
ideals of $\De_i$ (these are finite subsets $\Om\sb\De$ such that $x\le y,y\in\Om$
implies $x\in \Om$) given by the rule
$$(M,m)\mto \sets{u\in\De_i}{M_{\wt(u)}\ne0}.$$
\end{lmm}

In the next sections we will construct the weight functions
for the quivers induced by brane tilings and investigate
when the above assumptions are satisfied.

\section{Brane tilings and quiver potentials}\label{sec:brane}

\begin{dfn}
A bipartite graph $\bg=(\bg_0^+,\bg_0^-,\bg_1)$
consists of two sets of vertices $\bg_0^+,\bg_0^-$,
called the sets of white and black vertices respectively,
a set of edges $\bg_1$ and a map $\bg_1\arr\bg_0^+\xx\bg_0^-$.
The corresponding CW-complex is also denoted by $G$.
We denote the set of all vertices of \bg 
by $\bg_0=\bg_0^+\cup\bg_0 ^-$.
\end{dfn}

\begin{dfn}
A brane tiling is a bipartite graph $G$
together with an embedding of the corresponding CW-complex 
into the real two-dimensional torus $T$ so that the complement
$T\ms\bg$ consists of simply-connected components.
We identify any two homotopy equivalent embeddings.
The set of connected components of $T\ms\bg$ is denoted by 
$\bg_2$ and is called the set of faces of $G$.
\end{dfn}

We will always assume that the connected components of $T\ms\bg$
are convex polygons.
We define a quiver $\q=(\q_0,\q_1)$ dual to the brane tiling
\bg as follows. The set of vertices $\q_0$ is $\bg_2$, the set
of arrows $\q_1$ is $\bg_1$. For any arrow $a\in\q_1$ we define
its endpoints to be the polygons in $\bg_2$
adjacent to $a$.
The direction of $a$ is chosen in such a way that
the white vertex is on the right of $a$.
The CW-complex corresponding to $Q$ is automatically
embedded in $T$. 
The set of connected components of the complement,
called the set of faces of \q,
will be denoted by $\q_2$. It can be identified
with $\bg_0$. There is a decomposition $\q_2=\q_2^+\cup\q_2^-$
corresponding to the decomposition $\bg_0=\bg_0^+\cup\bg_0 ^-$.
It follows from our definition that the arrows
of the face from $\q_2^+$ go clockwise and the arrows 
of the face from $\q_2^-$ go anti-clockwise.

For any face $F\in\q_2$, we will denote by $w_F$
the necklace (equivalence class of cycles in \q
up to shift) obtained by going along the arrows of $F$.
We define the potential of $Q$ (see e.g. \cite{Ginz1,Bock1}
for the relevant definitions) by
$$W=\sum_{F\in\q_2^+}w_F-\sum_{F\in\q_2^-}w_F.$$
We want  to apply the results of Section \ref{sec:hilbert_schemes} 
to the algebra $\cC Q/(\dd W)$, which we call a quiver potential
algebra.

Consider the complex
$$\cZ^{\q_2}\arr^{d_2}\cZ^{\q_1}\arr^{d_1}\cZ^{\q_0},$$
where $d_2(F)=\sum_{a\in F}a$ for $F\in\q_2$ and
$d_1(a)=s(a)-t(a)$ for $a\in\q_1$. Its homology
groups are isomorphic to $H_*(T,\cZ)$.
We define the group \La as
$$\La=\cZ^{\q_1}/\angs{d_2(F)-d_2(F')}{F,F'\in\q_2}$$
and define the weight function $\wt:\cZ^{\q_1}\arr\La$
to be the projection. 
It is clear that the ideal 
$(\dd W)\sb\cC Q$ is automatically \La-homogeneous.
For any path $u$ in \q, we define its content $|u|\in\cZ^{\q_1}$
by counting the multiplicities of arrows in $u$. We define
the weight of the path $u$ by $\wt(u):=\wt(|u|)$.
Let $\ub\om:=\wt(d_2(F))$ for some (any) $F\in\q_2$.
The weight of an arrow $a\in\q_1$ will also usually be denoted by $a$.

\begin{lmm}\label{lmm:free group}
Assume that one of the following conditions is satisfied
\begin{enumerate}
	\item There exists at least one perfect matching of $G$.
	\item All faces of $\q$ contain the same number of arrows.
\end{enumerate}
Then the group \La is free.
\end{lmm}
\begin{proof}
Let us show first that $\coker d_2=\cZ^{\q_1}/\im d_2$
is free.
Let $\la=\sum\al_a a\in\cZ^{\q_1}$ be such that 
$k\la\in\im d_2$ for some $k\ge1$.
We have to show that $\la\in\im d_2$.
We have $d_1(k\la)=0$ and therefore $d_1(\la)=0$.
Let $\pi:\tl T\arr T$ be the universal covering of the torus $T$
and let \pq be the corresponding periodic quiver, 
which is the inverse image of \q in $\tl T$.
We will consider paths in \q and \pq, consisting of
arrows and their inverses.
The condition $d_1(\la)=0$ implies that we can construct a 
weak cycle $u$ in \q with content \la (every arrow is counted with
a plus sign and its inverse with a minus sign).
We can lift $u$ to some weak path $\tl u$ in \pq.
The condition $k\la\in\im d_2$ implies that $\tl u$ is
actually a  cycle. Its content \la can be represented
as a sum of $\pm d_2(F)$ for faces $F$ contained inside the cycle.

We note that
there is an exact sequence
$$\cZ\arr\La\arr\coker d_2\arr0,$$
where $\cZ\arr\La$ is given by $1\mto\ub\om=\wt(d_2(F))$
for some $F\in\q_2$. To show that \La is free we
need to show that the first map is injective.
Assume that $k\ub\om=0$ in \La for some $k\ge1$.
Then $kd_2(F_0)=\sum_{F\in\q_2}\la_Fd_2(F)$ for some $F_0\in Q_2$
and integers $\la_F$ ($F\in\q_2$) with $\sum\la_F=0$.

If there exists some perfect matching $I$ of \bg
then $\sum_{a\in I}x_a=\sum_{F\in\q_2}\la_F$.
There exists some $a\in F_0$ that is contained in $I$ and therefore
$x_a>0$ and $\sum_{F\in\q_2}\la_F>0$, contradicting 
our assumption.

If 	all faces of $\q$ contain the same number of arrows, say $r$,
then $r\sum_{a\in Q_1}x_a=\sum_{F\in\q_2}\la_F$.
For any $a\in F_0$, we have $x_a>0$ and therefore
$\sum_{F\in\q_2}\la_F>0$, contradicting our assumption.
\end{proof}

\begin{dfn}
A bipartite graph (dimer model) is called non-degenerate if all of its edges 
belong to some perfect matching.
\end{dfn}


The following result ensures that the second assumption
of Section \ref{sec:hilbert_schemes} is satisfied.

\begin{lmm}
Assume that one of the following conditions is satisfied
\begin{enumerate}
	\item The bipartite graph $G$ is non-degenerate.
	\item All faces of $\q$ contain the same number of arrows.
\end{enumerate}
Then for any $(x_a)_{a\in Q_1}\in\cN^{Q_1}\ms\set0$,
we have $\sum_{a\in Q_1} x_aa\ne0$ in \La. 
\end{lmm}
\begin{proof}
Assume that $\sum_{a\in\q_1}x_aa\in\cZ^{\q_1}$ is zero in \La,
where all $x_a\ge0$ and some of them are nonzero.
This means 
$$\sum x_aa=\sum_{F\in \q_2}\la_Fd_2(F),$$
for some $\la_F\in\cZ$ with $\sum\la_F=0$. 

If $G$ is non-degenerate, then for any 
$b\in Q_1$ with $x_b>0$ we can find 
a perfect matching $I$ of \bg containing $b$. 
Then $\sum_{a\in I}x_a=\sum_{F\in\q_2}\la_F$ and the left hand
side is strictly positive, as $x_b>0$ and $b\in I$. 
This contradicts the assumption $\sum_{F\in\q_2}\la_F=0$.

If 	all faces of $\q$ contain the same number of arrows, say $r$,
then $r\sum_{a\in Q_1}x_a=\sum_{F\in\q_2}\la_F$.
But the left hand side of this equation is non-zero and this contradicts
the assumption $\sum_{F\in\q_2}\la_F=0$.
\end{proof}

It follows that the arrows of $Q$ generate a strongly convex cone
in $\La_\cR=\La\ts_\cZ\cR$.
From now on we will always assume that the bipartite graph
is non-degenerate.

\begin{rmr}
The main result of Ishii and Ueda \cite{Ishii} relies on the existence
of some map $R:Q_1\arr \cR_{>0}$, such that
for any face $F\in Q_2$, we have $\sum_{a\in F}R(a)=2$.
If the bipartite graph $G$ is non-degenerate, then
the above lemma allows us to construct a linear map
$R:\La_\cR\arr\cR$ such that $R(a)>0$ for every arrow $a$ and
$R(\ub\om)=2$. This gives a map required in \cite{Ishii}.
\end{rmr}


\section{Path groupoid of a quiver potential}\label{sec:groupoid}
By a weak path in a quiver \q we will mean a path consisting of
arrows of a quiver and their inverses (for any arrow $a$ we
identify $aa\inv$ and $a\inv a$ with trivial paths). 
The usual paths will be sometimes called strict paths.
We will show that the potential $W$ from the previous section
defines an equivalence relation on the set of strict paths
(it will also be called the strict equivalence relation)
and on the set of weak paths (it will also be called the weak equivalence
relation). It is possible that strict paths are weakly equivalent
but not strictly equivalent.
The goal of this section is to prove Proposition
\ref{prp:weak equivalence} stating that two paths with
the same startpoints are weakly equivalent if and only if
their weights are equal.
In Lemma \ref{lmm:weak and strict} we investigate
when the weak equivalence of paths implies the strict equivalence.
This is needed in order to satisfy the first assumption of
section \ref{sec:hilbert_schemes}.

Let \pc be the category of paths of a quiver $\q$ 
(objects are vertices of $\q$ and morphisms are paths in $\q$)
and let \pg be the groupoid of weak paths of \q.

For any arrow $a\in\q_1$ we have $\dd W/\dd a=u-v$ for some paths
$u,v$. We denote the pair $(u,v)$ also by $\dd W/\dd a$ and consider
it as a relation on the set of paths. We denote by $\dd W$ the set of
all relations obtained in this way.

Let \fc (resp. \fg) be the factor category (resp. factor groupoid)
with respect to the equivalence relation $(\dd W)$ 
generated by $\dd W$  
(this is the minimal equivalence relation containing $\dd W$ 
and such that $(u,v)\in(\dd W)$ implies $(xuy,xvy)\in(\dd W)$
for any morphisms $x,y$ of  \pc (resp. \pg) 
with $t(y)=s(u)=s(v)$, $s(x)=t(u)=t(v)$).

\begin{rmr}[Quotients of groupoids]
We can formalize the operation of taking the quotient by an
equivalence relation in a groupoid with the help of
quotients by normal subgroupoids.
Let $G$ be a groupoid. 
We call the set of groups $H=(H_i\sb G(i,i))_{i\in\Ob G}$ 
a normal subgroupoid of $G$ if for any $h\in H_i$
and $x\in G(i,j)$, the element $xhx\inv$ belongs to $H_j$.
One can then construct the quotient groupoid $G/H$
in the usual way, namely $x,y\in G(i,j)$ are equivalent
if and only if $x\inv y\in H_i$. In the above example,
for any arrow $a\in \q_1$, there is a pair of cycles
$(au,av)$ in $W$, and we take the minimal normal
subgroupoid generated by the elements 
$u\inv v=(au)\inv(av)$.
\end{rmr}

\begin{lmm}
Any two cycles of $W$ with the same endpoints
are equal in \fc and \fg.
\end{lmm}
\begin{proof}
The cycles of $W$ starting at $i\in\q_0$ correspond to
faces incident to $i$. We may assume that our two 
faces have a common arrow incident to $i$.
This means that the corresponding cycles have either equal first arrows
or equal last arrows. Let us assume that they have
the same first arrow $a$. Then they are of the form
$ua$ and $va$. The set of relations $\dd W$
contains the pair $(u,v)$. Therefore $(ua,va)\in(\dd W)$.
\end{proof}

\begin{rmr}
The groupoid \fg can be also constructed as
a quotient of \pg modulo the relation $u\sim v$
for any two cycles $u,v$ of $W$ with the same endpoints.
\end{rmr}

Let $\pi:\tl T\arr T$ be the universal covering of a torus
and let \pq be the inverse image of \q, called a
periodic quiver.
For any point $i\in\pq_0$ and any path $u$ in $\q$ with
$s(u)=\pi(i)$ there is a unique lifting $u'$ of $u$ in \pq
with $s(u')=i$.
Analogously to the weak equivalence in \q, 
we introduce an equivalence relation on the set of
weak paths in \pq. It is generated by the pairs $(u,v)$, where
$u,v$ are the cycles along the faces of \pq  
having the same endpoints.
An equivalence class of the cycles along the faces 
with the endpoint $i$ will be denoted by $\om_i$. 
We will write just $\om$ if we do not want to specify the 
endpoint of the cycle.

\begin{lmm}
For any weak path $u:i\arr j$ we have $u\om_i\sim\om_j u$.
\end{lmm}
\begin{proof}
We have to prove the claim just when $u=a$, where $a:i\arr j$ is an arrow.
Let $wa$ be some cycle along the face of $Q$.
Then $wa\sim\om_i$ and $aw\sim\om_j$. It follows that 
$a\om_i\sim awa\sim\om_j a$. 
\end{proof}

\begin{lmm}
Let $u,v$ be two weak paths in \pq. 
Then $u\sim v$ if and only if they are equivalent in \q 
(that is, $\pi(u)\sim\pi(v)$) and have the same startpoints.
\end{lmm}

\begin{lmm}\label{lmm:cycle in pq}
Any weak cycle in \pq is equivalent to $\om^k$ for some $k\in\cZ$.
\end{lmm}
\begin{proof}
Assume that $u$ is a weak cycle without self-intersections.
Then it can be written as a product of cycles (taken in right
direction) along the faces contained in $u$. But this product
is a power of \om. In the general case, we can find a subcycle without
self-intersections, represent it as a power of \om, and move it to the
end of the cycle using the fact that \om commutes with paths. Then we repeat
our procedure. 
\end{proof}

For any weak path $u$ in \q we define its content $|u|\in\cZ^{\q_1}$
by counting every arrow with a plus sign and its inverse with a
minus sign. As for strict paths, we define $\wt(u):=\wt(|u|)$.
For any weak path $u$ in \pq, we define its weight by
$\wt(u):=\wt(\pi(u))\in\cZ^{\q_1}$. Recall that $\wt(\om)=\ub\om$.

\begin{lmm}
Assume that $u$ is a weak path in \pq such that $\wt(u)=0$.
Then $u$ is a weak cycle and it is equivalent to the trivial path.
\end{lmm}
\begin{proof}
Any arrow in \pq defines a vector in the plane, and two
arrows $a,b$ define the same vector if $\pi(a)=\pi(b)$.
It follows that the vector between $s(u)$ and $t(u)$ is
determined by $\pi(u)$, or just $|\pi(u)|$.
We can write $|\pi(u)|=\sum_{F\in\q_2}\la_Fd_2(F)\in\cZ^{\q_1}$, where $\sum\la_F=0$.
Any $d_2(F)$ determines the zero vector in the plane. 
It follows that $|\pi(u)|$ determines the zero vector in the plane
and therefore $u$ is a cycle.
According to Lemma \ref{lmm:cycle in pq}, $u\sim\om^k$ for some $k\in\cZ$.
Then $\wt(u)=\wt(\om^k)=k\ub\om=0$ and by the proof
of Lemma \ref{lmm:free group}, we have $k=0$.
\end{proof}

\begin{prp}\label{prp:weak equivalence}
Weak paths in $Q$ (or in \pq) having the same startpoints
are equivalent if and only if their weights are equal.
\end{prp}
\begin{proof}
Follows immediately from the previous lemma.
\end{proof}

\begin{rmr}
The above result is very similar to the first assumption 
of Section \ref{sec:hilbert_schemes}. The only difference is that
paths in \q having the same weight and the same startpoint are 
weakly equivalent but not necessarily strictly equivalent. 
\end{rmr}


\begin{rmr}\label{rem:description of De}
For any two nodes $i,j\in\pq_0$ and any two weak paths $u,v$
between them we have $v\sim u\om^k$ for some $k\in\cZ$. 
It follows that there exists a shortest strict path $v_{ij}$ between
$i$ and $j$, i.e.\ a path such that any other strict path between
$i$ and $j$ is weakly equivalent to $v_{ij}\om^k$, $k\ge0$.
It follows that the set of weak equivalence classes of strict paths starting
at some fixed point $i_0\in\pq_0$ can be identified with $\pq\xx\cN$, and
the set of weak equivalence classes of weak paths starting at $i_0$
can be identified with $\pq\xx\cZ$.
\end{rmr}

\begin{lmm}\label{lmm:weak and strict}
The following conditions are equivalent
\begin{enumerate}
	\item Given two paths $u,v$ with the same endpoints and an arrow $a$
	with $s(a)=t(u)$, if $au$ is strictly equivalent to $av$ then
	$u$ is strictly equivalent to $v$. 
	\item Given two paths $u,v$ with the same endpoints and an arrow $a$
	with $t(a)=s(u)$, if $ua$ is strictly equivalent to $va$ then
	$u$ is strictly equivalent to $v$. 
	\item The map $\fc\arr\fg$ is injective.
	\item Two paths in \q with the same startpoints
	are strictly equivalent if and only if their weights are equal.  
\end{enumerate}
\end{lmm}
\begin{proof}
$(1)\imp(2)$. Assume that two paths $u,v$ have the same endpoints
$i,j$ and $ua\sim va$ for some arrow $a:k\arr i$. We can find a path $w$
such that $aw\sim \om_i$. 
Then 
$$\om_ju\sim u\om_i\sim uaw\sim vaw\sim v\om_i\sim\om_jv.$$
Using condition $(1)$ we obtain $u\sim v$.\\
$(2)\imp(3)$. 
Let us give a new description of the groupoid
$\fg$. It is obtained from the category \pc by adding
morphisms $t_i:i\arr i$, $i\in  Q_0$,
such that $at_i=t_ja$ for any arrow $a:i\arr j$ and 
$t_iw=1$ for any cycle $w$ along a face starting at $i\in Q_0$. 
It follows that two strict paths $u,v$ with the same endpoints
are weakly equivalent if and only if $v$ is obtained from $u$
by first inserting words of the form $t_iw$ in $u$ (say $r$ times), 
then moving the $t_i$'s inside the obtained word, and finally deleting $r$
words of the form $t_iw$. 
Equivalently, we can insert $r$ cycles along faces into $u$ and $v$
so that we obtain equal words. This implies that $u\om^r$
and $v\om^r$ are strictly equivalent. Condition $(2)$ now implies
that $u$ and $v$ are strictly equivalent.\\
$(3)\imp(4)$. This follows from Proposition \ref{prp:weak equivalence}.\\
$(4)\imp(1)$. Clear.
\end{proof}

We will call the following condition the first consistency 
condition and we will assume it throughout the paper.

\begin{condition}\label{cond:B}
The equivalent conditions of Lemma \ref{lmm:weak and strict} are satisfied. 
\end{condition}

\begin{rmr}\label{rmr:R-charge}
It follows from \cite[Lemma 5.3.1]{HHV} and \cite{Broomhead1}
that the condition $(4)$ of Lemma \ref{lmm:weak and strict}
is satisfied if there exists an $R$-charge
on the brane tiling. An $R$-charge is a collection 
$(R_a)_{a\in Q_1}\in (0,1)^{Q_1}$ such that for every
face $F\in Q_2$ we have
$$\sum_{a\in F}R_a=2$$
and for every node $i\in Q_0$ we have
$$\sum_{a\ni i}(1-R_a)=2.$$
An easily verified criterion for the existence of $R$-charges
is given in \cite{KenyonSchlenker}.
We thank Alastair King for this remark.
\end{rmr}

Under the above condition and the non-degeneracy of the brane tiling,
all the assumptions of Section \ref{sec:hilbert_schemes} are satisfied.
Then Lemma \ref{lmm:La-graded mod. and ideals} implies

\begin{prp}\label{prp:formula1}
We have
$$Z^i(A)=\sum_{\stackrel{\Om\sb\De_i}{\text{fin.ideal}}}
\prod_{u\in\Om}x_{t(u)}\in\cQ\pser{x_j|j\in\q_0}.$$
\end{prp}

This formula allows the computation of $Z^i(A)$ with the help
of a computer. 
Let us discuss some examples that were considered earlier in the literature.

\begin{exm}
The basic example is $\cC^3$. 
\begin{figure}[h!]%
{\includegraphics[height=4cm]{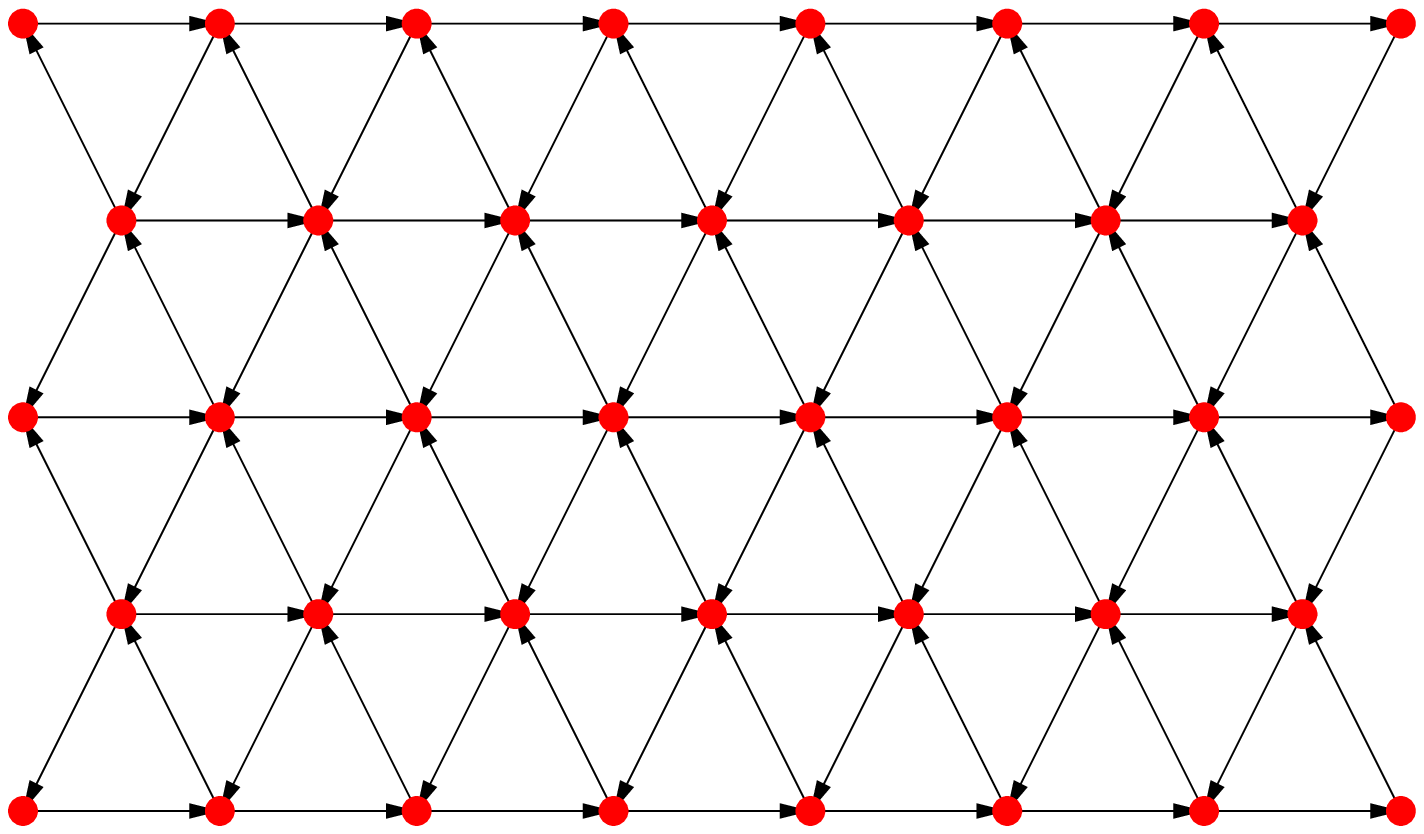}}%
\llap{\includegraphics[height=4cm]{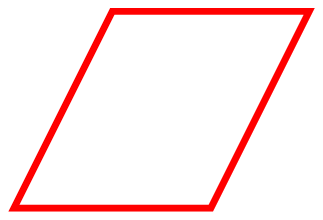}}%
\caption{The periodic quiver and a fundamental domain for $\cC^3$.}
\end{figure}
It corresponds to the quiver with one node $0$,
three loops $x,y,z$, and potential $W=xyz-xzy$. The quiver potential
algebra is $A=\cC Q/(\dd W)\iso\cC[x,y,z]$. Every path in $Q$ is equivalent
to $x^ky^lz^m$ for some $(k,l,m)\in\cN^3$ and therefore the poset of 
paths $\De_0$ can be identified with $\cN^3$. Ideals in $\De_0$ 
correspond to the $3$-dimensional Young diagrams.
The generating function $Z^0(A)$ was known already to MacMahon.
\end{exm}

\begin{exm}
Szendr\H oi \cite{Szen1} considered the case of the conifold.
\begin{figure}[h!]%
{\includegraphics[height=4cm]{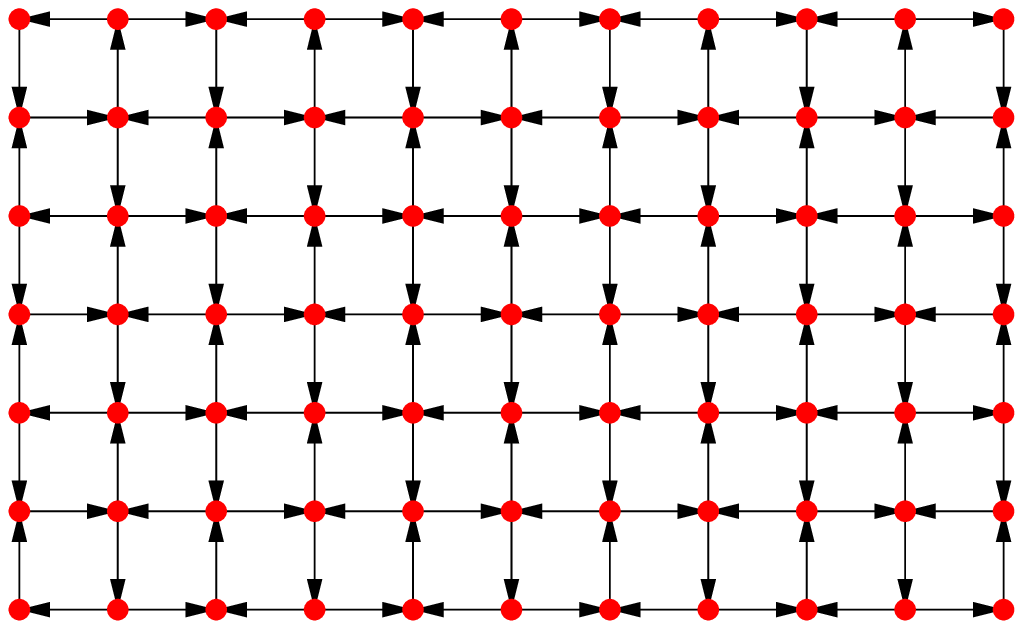}}%
\llap{\includegraphics[height=4cm]{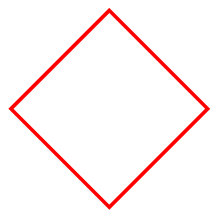}}%
\caption{The periodic quiver and a fundamental domain for the conifold.}
\end{figure}
It corresponds to the quiver with nodes $0,1\in\cZ_2$ and arrows $x_i:i\to i+1$,
$y_i:i+1\to i$, $i=0,1$.  The potential is given by
$$W=x_0x_1y_1y_0-x_1x_0y_0y_1.$$
The poset $\De_0$ corresponds to the pyramid arrangement from \cite{Szen1}.
Ideals in $\De_0$ correspond to the pyramid partitions from \cite{Szen1}.
A nice closed formula for $Z^i(A)$ was conjectured by Szendr\H oi and
proved by Young \cite{Young1}.
\end{exm}

\begin{exm}
The orbifold $\cC^3/\cZ_n$ with a group action $\frac 1n(1,0,-1)$. 
\begin{figure}[h!]%
{\includegraphics[height=4cm]{c33x.eps}}%
\llap{\includegraphics[height=4cm]{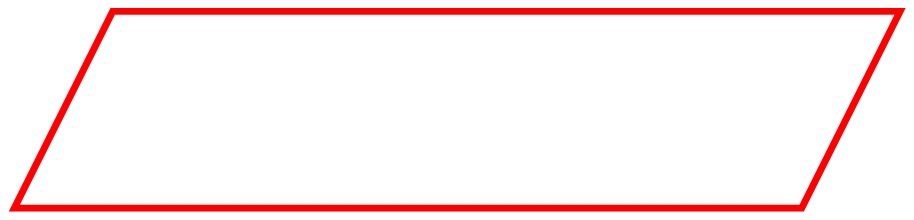}}%
\caption{The periodic quiver and a fundamental domain for $\cC^3/\cZ_4$.}
\end{figure}
The corresponding Mckay quiver is given by
$$\xymatrix{
0\ar@/^/[r]^{x_0}\ar@(ul,ur)^{z_0}[]&
1\ar@/^/[r]^{x_1}\ar@(ul,ur)^{z_1}[]\ar[l]^{y_0}&
2\ar@/^/[r]^{x_2}\ar@(ul,ur)^{z_2}[]\ar[l]^{y_1}&
\dots\ar@/^/[r]^{x_{n-2}}\ar[l]^{y_2}&
n-1\ar@{<->}@/^2.3pc/[llll]_{x_{n-1},y_{n-1}}\ar@(ul,ur)[]^{z_{n-1}}\ar[l]^{y_{n-2}}}$$
The potential is given by
$$W=\sum_{i=1}^n(x_iy_iz_{i+1}-z_{i}y_ix_i).$$
The poset of paths $\De_0$ can be identified with the lattice $\cN^3$. 
New arrow $x$ means increasing of $x$-coordinate,
$y$ -- increasing of $y$-coordinate, $z$ -- increasing of 
$z$-coordinate. The ideals
in $\De_0$ correspond to 
$3$-dimensional Young diagrams. 
The partition function $Z^0(A)$ is now an element
of $\cZ\pser{x_0,\dots,x_{n-1}}$.
A simple formula for $Z^0(A)$ was also found by Young
\cite{Young1}.
\end{exm}
\section{Path poset and perfect matchings}\label{sec:perf mat}
Let us first recall the nice correspondence between $3$-dimensional
Young diagrams and perfect matchings of the plane tiling by equilateral
triangles
(we use the dual point of view and call the set of edges a perfect
matching if every face contains exactly one edge from this set).
The following picture should explain this correspondence

\begin{figure}[h!]%
\hfill%
{\includegraphics[height=4cm]{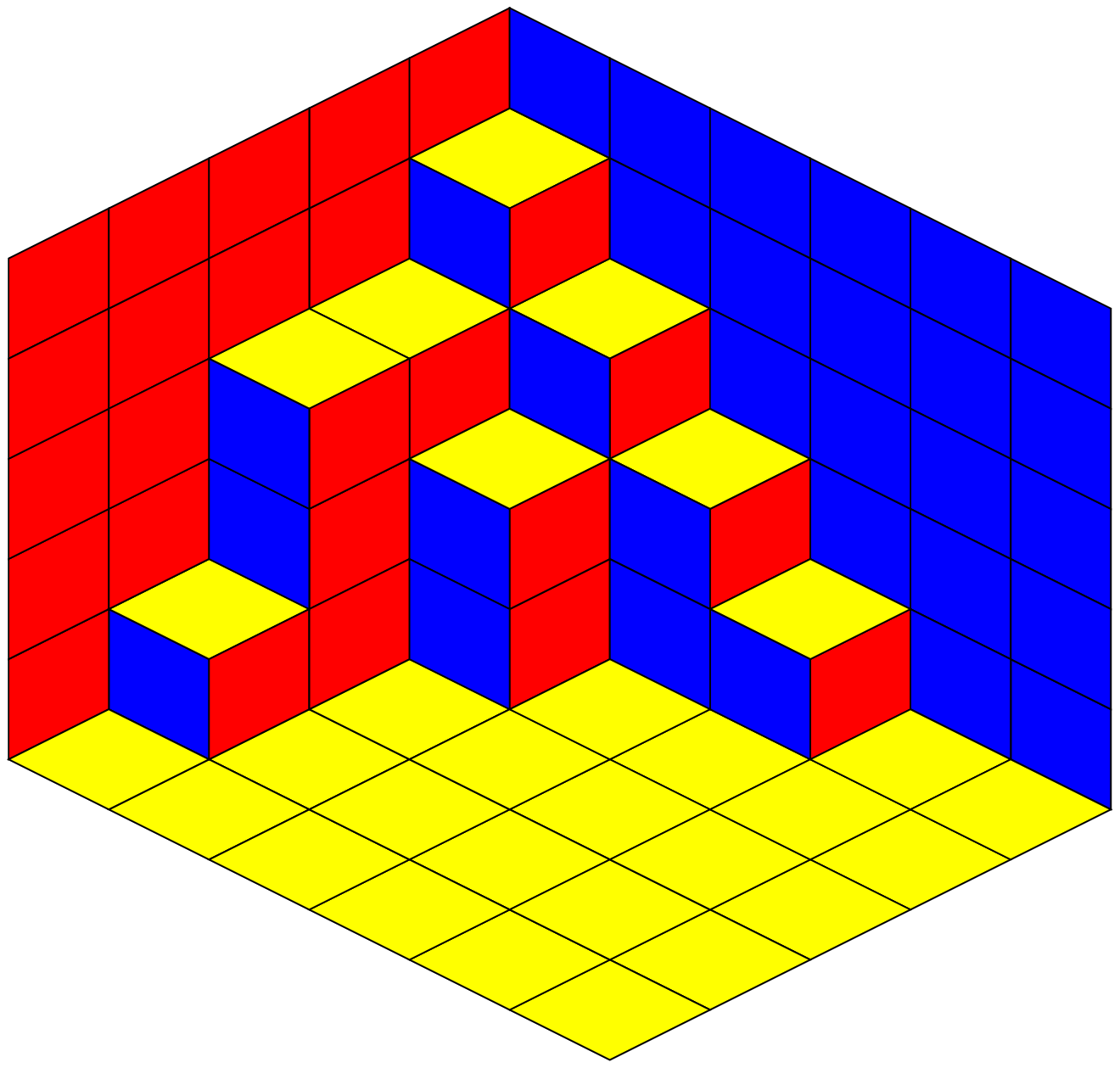}}%
\hfill%
{\includegraphics[height=4cm]{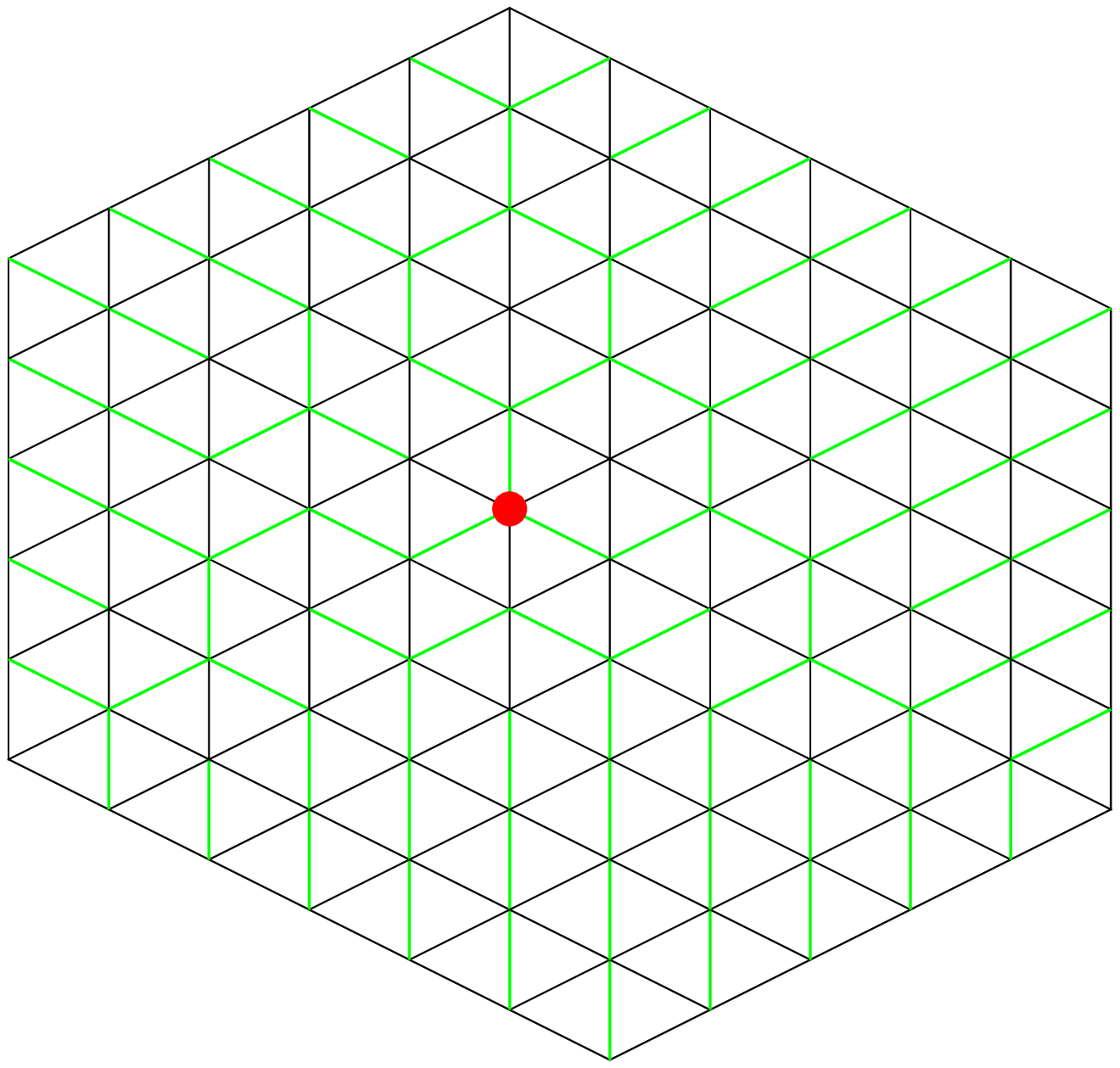}}%
\hfill{}
\caption{The right picture is obtained from the left 
by drawing the shortest diagonals of the small parallelograms in green.}
\end{figure}

As we have seen in the last section, the $3$-dimensional Young
diagrams correspond to the ideals of the poset $\De_0$ in the case
of a quiver $Q$ with one node $0$, three arrows $x,y,z$ and potential
$W=xyz-zyx$. The quiver $Q$ is embedded in the torus and the corresponding 
periodic quiver $\pq$  is embedded in the plane 
and defines there precisely the plane tiling by triangles
discussed above. 

We have thus all ingredients to generalize the above correspondence 
to arbitrary brane tilings. On the one hand we have ideals in the
poset of equivalence classes of paths, and on the other hand
we have perfect matchings of the plane tiling defined by the 
periodic quiver. The goal of this section is to construct such
a correspondence.

So, let $(Q,W)$ be a quiver potential arising from a brane tiling
satisfying the conditions of the previous section.
We fix some $i_0\in\pq_0$. 
Let $\De=\De_{i_0}$ be the poset of equivalence classes of paths
in \pq starting at $i_0$ and let $\De'$ be the poset
of equivalence classes of weak paths starting at $i_0$.

\begin{rmr}
We can identify $\De$ with the set of equivalence classes 
of paths in \q starting at $\pi(i_0)$.
\end{rmr}

For any $i\in\pq_0$ and any subset $\Om\sb\De'$,
we define $\Om_i\sb \Om$ to be the set of all weak
paths ending at $i$.
For any point $i\in\pq_0$, we denote by $v_i\in\De_i$ the shortest
path between $i_0$ and $i$ (see Remark \ref{rem:description of De}).
Then  
any weak path in $\De'_i$ is equivalent 
to $\om^kv_i$ for some $k\in\cZ$.
Therefore we can identify $\De'$ with $\pq_0\xx\cZ$ and
identify $\De$ with $\pq_0\xx\cN$.
For any ideal $\Om\sb\De$, we define an ideal 
$\Om':=\Om\cup(\pq_0\xx\cZ_{<0})$ in $\De'$.

\begin{prp}
For any finite ideal $\Om\sb\De$ the set
$$I(\Om):=\sets{a:i\arr j}{\exists u\in\Om'_i,au\not\in\Om'_j}\sb\pq_1$$
is a perfect matching of the plane tiling defined by \pq 
(i.e.\ of the dual bipartite graph).
\end{prp}
\begin{proof}
Let $F$ be some face in \pq.
We will show that some of its arrows are contained in $I(\Om)$.
If not, then for any arrow $a:i\arr j$ and any path $u\in\Om'_i$,
we have $au\in\Om'_j$. Going along the face $F$, we get $\om u\in\Om'_i$
and therefore $\om^ku\in\Om'_i$ for any $k\ge0$. It follows
that $\Om_i$ is infinite.

Assume that there are two arrows $a_1,a_2\in F$ contained in $I(\Om)$.
Let $w_2a_2w_1a_1$ be the cycle along the face $F$. Let $u_i\in\Om'_{s(a_i)}$
be such that $a_iu_i\not\in\Om'_{t(a_i)}$, $i=1,2$.
Then $w_1a_1u_1\not\in\Om'_{s(a_2)}$. 
This implies that $\om u_2\le w_1a_1u_1$ and $w_2a_2u_2\le u_1$.
Analogously $\om u_1\le w_2a_2u_2$.
This means that $\om u_1\le u_1$, a contradiction. 
\end{proof}

We call the perfect matching $I_0=I(\emptyset)$ the canonical perfect matching.
We have
$$I_0=\sets{a:i\arr j}{av_i>v_j}.$$
A perfect matching $I$ will be called congruent to $I_0$ if
the symmetric difference $(I\ms I_0)\cup(I_0\ms I)$ is finite. 
It is clear that any $I(\Om)$ is congruent to $I_0$.

The following condition will be called the second consistency condition
and will be assumed throughout the paper.

\begin{condition}\label{cond:C}
For any two vertices $i,j\in\pq_0$, there exists an arrow $a:j\arr k$
such that $av_{ij}$ is the shortest path between $i$ and $k$
(here $v_{ij}$ is the shortest paths between $i$ and $j$,
see Remark \ref{rem:description of De}).
Also the dual condition is satisfied.
\end{condition}

\begin{rmr}
It follows from the results of \cite{Broomhead1} that the above
condition is satisfied if there exists an $R$-charge
on the brane tiling (see Remark \ref{rmr:R-charge}). We thank
Alastair King for this remark.
\end{rmr}

\begin{rmr}
To verify the above condition we have to test just a finite number
of shortest paths. This can be seen as follows. 
One can show that a path $v$ in \pq is the shortest path if and only
if there exists some perfect matching of the bipartite graph $G$
that does not intersect $\pi(v)$. 
If the above condition is not satisfied then we can find some path
$v$ in $Q$ that does not intersect some perfect matching
and such that for any arrow $a$ starting at $t(v)$ the path $av$ intersects
every perfect matching. If the number of cycles in $v$ is greater than
the number of perfect matchings of $G$, we can subtract some
cycle from $v$ so that the new path $v'$ intersects the the same set of
perfect matchings. It follows that already the path $v'$ contradicts
the above condition. Therefore, we can assume that the number of 
cycles in the considered paths on $Q$ is bounded by the number of perfect matchings.
Such paths form a finite set.
\end{rmr}

In the next theorem we will actually need the above condition just for $i=i_0$.
We will see in the next section that the same condition is needed
in order to ensure that the the quiver potential algebra $\cC Q/(\dd W)$ is a
3-Calabi-Yau algebra.

\begin{thr}
For any perfect matching $I$ congruent to $I_0$, there exists
a unique finite ideal \Om of \De such that $I(\Om)=I$.
This ideal is defined as $\Om'\cap\De$, where $\Om'$ is the minimal
ideal in $\De'$, containing $\pq_0\xx\cZ_{<0}$ and such that
for any arrow $a\not\in I$ and any $u\in\Om'_{s(a)}$, we have
$au\in\Om'_{t(a)}$. 
\end{thr}
\begin{proof}
Let $\Om'$ be the ideal in $\De'$ described above and let $\Om=\Om'\cap\De$.
Let us show that $\Om$ is finite.
The union $I\cap I_0$ considered on the bipartite graph $\wtl\bg$
consists of coinciding arrows and a finite number of cycles.
On the periodic quiver \pq these cycles go through a finite
number of faces and intersect them alternatingly in arrows
from $I$ and $I_0$. 
We will call such a cycle in \pq a trap. Every trap is 
determined by a sequence $S$ of arrows in $I$ and $I_0$ that
it intersects.

We claim that every trap contains $i_0$. 
Assume that $i_0$ is outside of a trap.
Let $au\in\De$ be one of the shortest paths among all paths to the points
inside the trap. Then $s(a)$ is outside of the trap and $a\in S$. 
We have $v_{t(a)}=av_{s(a)}$, so $a\not\in I_0$.  
This implies that all arrows from $S\cap I_0$ point outside the trap
and all arrows from $S\cap I$ point inside the trap.
It follows that there exists some point $i$ in the trap
such that for any arrow $a:i\arr j$, we have $av_i>v_j$, which
contradicts our assumption. Indeed, if this is not the case
then using the fact that the number of points in the trap
is finite, we can find a cycle $a_r\dots a_1$ such that
$a_iv_{s(a_i)}\sim v_{t(a_i)}$ for all $i$ and therefore 
$a_r\dots a_1v_{s(a_1)}\sim v_{s(a_1)}$, a contradiction.
So, the point $i_0$ is in the trap. 

We claim that all
arrows from $S\cap I_0$ point inside the trap. Let
$au\in\De$ be one of the shortest paths among all paths to the 
points outside the trap. Then $av_{s(a)}=v_{t(a)}$,
so $a\not\in I_0$. It follows that $a\in I$ and all arrows 
from $S\cap I$ point outside the trap.

Starting with some path $u=(i,-1)\in \Om'$, we apply to it all
compatible arrows from $\pq_1\ms I$. 
It follows that either we stay in the same
plane, or we apply some arrow from $I_0\ms I$ increasing the
plane number by one and get automatically to some trap. 
Applying again the arrows from $\pq_1\ms I$,
we cannot get outside of the trap (all arrows pointing outside
the trap are contained in $I$). The only possibility to get
outside of the trap is to use the property that $\Om'$ is an ideal.
So, we can move outside of the trap along (the inverse of) some arrow from $I_0$,
but then the plane number will decrease by one.

We get an alternative description of $\Om$. 
The point $(i,k)\in\pq_0\xx\cN$ is contained in \Om if and only
if $i$ is contained in more than $k$ traps. From this description,
it follows that \Om is finite. 

Let us show that $I(\Om)=I$. We have to check two properties:
\begin{enumerate}
	\item If $a:i\arr j$ is in $I$ then there exists $u\in\Om'_i$ such that $au\not\in\Om'_j$.
	\item If $a:i\arr j$ is not in $I$ then for all $u\in\Om'_i$, we have $au\in\Om'_j$.
\end{enumerate}
The second property follows from our construction.
Consider some $a:i\arr j$ in $I$. Let $aw$ be a cycle along the face
in \pq. Then no arrow from $w$ is contained in $I$.
It follows that for any $u\in\Om'_j$, we have $wu\in\Om'_i$.
Let $u=(j,k)\in\Om'_j$ be the path with the maximal $k\ge-1$.
Then $wu\in\Om'_i$, but $awu=(j,k+1)\not\in\Om'_j$. 
This proves the first property.

Let us show that \Om is uniquely determined. Assume that there
exists some other finite ideal $\wtl\Om\sp\Om$ with $I(\wtl\Om)=I$.
Let 
$$A=\sets{i\in\pq_0}{\mbox{there exists } u\in\wtl\Om_i\ms\Om_i}.$$
We claim that for any arrow $a:i\arr j$ not in $I$, 
we have $i\in A$ if and only if $j\in A$.
If $i\in A$ then there exists some $u\in\wtl\Om_i\ms\Om_i$.
It follows that $au\in\wtl\Om_j\ms\Om_j$ and $j\in A$.
Conversely, assume that $j\in A$, so there exists some $k\ge0$
with $(j,k)\in\wtl\Om_j\ms\Om_j$.
If $a\not\in I_0$ then $a(i,k)=(j,k)$ and 
$(i,k)\in\wtl\Om_i\ms\Om_i$, so $i\in A$.
If $a\in I_0$ then $(j,k)$ is in some trap,
so $(j,0)\in\Om_j$ and therefore $k>0$. It follows
that $(i,k-1)=a\inv(j,k)\in\wtl\Om_i\ms\Om_i$
and $i\in A$. 

The set $\pq_1\ms I$ connects all the nodes of \pq,
so $A=\pq_0$. This contradicts the assumption
that $\wtl\Om$ is finite.
\end{proof}

For any perfect matching $I$ of \pq, consider its characteristic
function $\hi_I:\pq_1\arr\cZ$. If $I$ is congruent to $I_0$
then there exists a unique function $h_I:\pq_0\arr\cZ$,
called the height function, such that for every arrow $a:i\arr j$, we have
$$h_I(i)-h_I(j)=\hi_I(a)-\hi_{I_0}(a),$$
and $h_I(i)=0$ for $i$ far enough from $i_0$.
It is clear that $h_I(i)$ equals the number of traps (introduced in the theorem)
containing $i$.
In the course of the proof of the theorem, we have found an alternative
description of the finite ideal \Om satisfying $I(\Om)=I$. Namely,
$$\Om=\sets{(i,k)}{0\le k\le h_I(i)-1}.$$

\begin{crl}
We have
$$Z^{i_0}(A)=\sum_{I-\text{perf.mat.}}\prod_{i\in\pq_0}x_{\pi(i)}^{h_I(i)}.$$
\end{crl}

\section{Calabi-Yau property}\label{sec:CY}
Let us study condition \ref{cond:C} in more detail.
We use the dual formulation.

\begin{lmm}\label{lmm:cond:C}
The following conditions are equivalent
\begin{enumerate}
	\item If $u$ is the shortest path in \pq (between its endpoints), then
	there exists some arrow $a$ with $t(a)=s(u)$ such that $ua$
	is also the shortest path.
	\item If $w$ is a weak path in \pq (or in \q) such that $wa$ 
	is equivalent to a strict path for any
	arrow $a$ with $t(a)=s(w)$ then $w$ is equivalent to a strict path.
\end{enumerate}
\end{lmm}
\begin{proof}
Assume that the first condition is satisfied and let $w$ be as in the second condition.
We can write $w=\om^k u$, where $u$ is the shortest path. To show that
$w$ is equivalent to a strict path, we have to prove that $k\ge0$.
Assume that $k<0$. By the first condition there exists some arrow $a$
with $t(a)=s(u)$ such that $ua$ is the shortest path. But this
implies that $wa=\om^kua$ is not the strict path, as $k<0$. This
contradicts our assumption on $w$.

Assume that the second condition is true and let $u$ be as in the first condition.
We consider the weak path $w=\om\inv u$. If $ua$ is not the shortest path for some 
arrow $a$, then $\om\inv ua=wa$ is equivalent to a strict path. It follows from the second
condition that $w=\om\inv u$ is equivalent to a strict path, and therefore $u$
is not the shortest one, a contradiction.
\end{proof}

For any \La-graded $A$-module $M$ and for any $\la\in\La$, we denote
by $M[\la]$ a new \La-graded $A$-module with $M[\la]_\mu=M_{\mu+\la}$.
For any point $i\in Q_0$ we denote by $S_i$ the one-dimensional 
$A$-module concentrated at $i$. We endow $S_i$ with a \La-grading 
so that all its elements have degree zero.

\begin{prp}\label{prp:resolution}
Assume that all consistency conditions are satisfied.
Then for any $i\in Q_0$ there exists an exact sequence of \La-graded $A$-modules
$$0\arr P_i[-\ub\om]
\arr^{\cdot b}\bigoplus_{b:k\arr i}P_k[b-\ub\om]
\arr^{\cdot\frac{\dd W}{\dd b}a\inv}\bigoplus_{a:i\arr j}P_j[-a]
\arr^{\cdot a}P_i\arr S_i\arr 0.$$
\end{prp}
\begin{proof}
Exactness of 
$$\bigoplus_{b:k\arr i}P_k[b-\ub\om]
\arr^{\cdot\frac{\dd W}{\dd b}a\inv}\bigoplus_{a:i\arr j}P_j[-a]
\arr^{\cdot a}P_i\arr S_i\arr 0.$$
is known (see e.g. \cite{Bock1}).
We just have to prove the exactness in the term $\oplus P_k[b-\ub\om]$.
A basis of $P_k$ is given by the set $\De_k$ of equivalence classes 
of paths starting at $k$. So a general element of
$\bigoplus_{b:k\arr i}P_k[b-\ub\om]$
can be uniquely written in the form
$$f=\sum_{t(b)=i}\sum_{u\in\De_{s(b)}}x_{b,u}u\ts b^*,$$
where $x_{b,u}\in\cC$, and we use the symbol $b^*$ 
to denote the appropriate direct summand.
Let us analyze the condition that for $a:i\arr j$,
the image of $f$ in $P_j[-a]$ is zero.
Let the terms of $W$ that contain $a$ be $abv$ and
$acw$. Then we have 
$$\sum_{u\in\De_{s(b)}}x_{b,u}uv=\sum_{u'\in\De_{s(c)}}x_{c,u'}uw,$$ 
or, equivalently
$$\sum_{u\in\De_{s(b)}}x_{b,u}ub\inv=\sum_{u'\in\De_{s(c)}}x_{c,u'}u'c\inv.$$ 
It follows that if $x_{b,u}\ne0$ then $ub\inv=u'c\inv$ for some path 
$u'\in\De_{s(c)}$ and $x_{c,u'}=x_{b,u}$.
We have shown this for two arrows $b,c$ with $t(b)=t(c)=i$ that belong to 
adjacent faces. If $f$ is mapped to zero, then this is true
for two arbitrary arrows $b,c$ with $t(b)=t(c)=i$. If $x_{b,u}\ne0$,
then the weak path $\ga=ub\inv$ satisfies the condition that for any
arrow $c$ with $t(c)=s(\ga)=t(b)=i$, the weak path $\ga c$ is a strict path.
From Lemma \ref{lmm:cond:C} it follows that \ga is a strict path. 
Then $x_{b,u}\ga\in P_i[-\ub\om]$ is mapped to
$x_{b,u}u\ts b^*$ in $P_k[b-\om]$.
This shows the exactness of the sequence in the second term.
\end{proof}

\begin{thr}
Assume that all consistency conditions are satisfied.
Then the quiver potential algebra $\cC Q/(\dd W)$ is a 3-Calabi-Yau algebra.
\end{thr}
\begin{proof}
We have to show the exactness of the sequence \cite[Prop. 5.1.9, Cor. 5.3.3]{Ginz1}
$$0\arr\bigoplus_i e_iA\ts Ae_i\arr^j\bigoplus_{b:k\arr i}Ae_k\ts e_iA\arr
\bigoplus_{a:i\arr j}Ae_j\ts e_iA\arr\bigoplus_i Ae_i\ts e_iA\arr A\arr 0.$$
The exactness should be actually shown just in the first and the second terms.

Let us show that $j$ is injective. The  map $j$ is given  by \cite{Ginz1}
$$x\ts y\mto \sum_{b:k\arr i}yb\ts x-y\ts bx.$$
A general element in $\bigoplus_i e_iA\ts Ae_i$ can be written
in the form 
$$\sum_{i\in Q_0}\sum_{u\in \De_i}x_u\ts u,$$ 
where $x_u\in e_i A$ for $u\in \De_i$. If this element maps to zero
then for any arrow $b:k\arr i$, we have
$$\sum_{u\in\De_i}ub\ts x_{u}=\sum_{u\in\De_k}u\ts bx_u,$$
and therefore $x_u=bx_{ub}$ for any $u\in\De_i$.
This easily implies that all $x_u$ are zero.

To prove the exactness in the second term, we just have to show that the
character of the complex of \La-graded modules is zero. The character is defined
as follows. For any \La-graded
vector space $V$ that is finite-dimensional in any degree, 
we define $\ch V=\sum_{\la\in\La}\dim V_\la t^\la$. 
For any finite complex $V_\cdot$ of \La-graded vector spaces, 
we define $\ch(V_\cdot)=\sum_{k\in\cZ}(-1)^k\ch(V_k)$).
For a fixed $i\in Q_0$ the corresponding summand of every component
of the complex is obtained by tensoring the component of the previous
proposition with $e_iA$. As the complex in the proposition is exact,
the character of the i-th component is zero.
\end{proof}

\begin{rmr}
For any finite-dimensional modules $M,N$ over $\cC Q/(\dd W)$,
we have
$$\Ext^k(M,N)\iso\Ext^{3-k}(N,M)\dual\qquad k=0,1,2,3.$$
If $M$ and $N$ are \La-graded modules, then we have
an isomorphism of \La-graded vector spaces
$$\Ext^k(M,N)\iso \Ext^{3-k}(N,M[-\ub\om])\dual\qquad k=0,1,2,3.$$
\end{rmr}
\section{Donaldson-Thomas invariants}\label{sec:dt invariants}
Let $(Q,W)$ be a quiver potential arising from a brain tiling
satisfying all consistency conditions and let $A=\cC Q/(\dd W)$. 
Let $i_0\in\q_0$ and $\al\in\cN^{Q_0}$. 
We put $X=\Hilb^\al_{i_0}(A)$. It is known that $X$ has
an (equivariant) symmetric obstruction theory 
\cite[Theorem 1.3.1]{Szen1}.
According to Behrend and Fantechi \cite[Theorem 3.4]{BehrFant1},
the DT type invariant
of $X$ can be computed by the formula
$$\#^{vir}X=\sum_{(M,m)\in {X^T}}(-1)^{\dim T_{(M,m)}X},$$
where $T_{(M,m)}X$ is the tangent space at the $T$-fixed 
point $(M,m)$.
For any finite ideal $\Om\sb \De_{i_0}$, let $(M,m)$ be the
corresponding $T$-fixed $i_0$-cyclic $A$-module
and let $d_\Om$ be the dimension of the tangent space at this
point. Then we have
$$Z_{DT}^{i_0}(A)=\sum_{\Om\sb \De_{i_0}}(-1)^{d_\Om}x^{\ub\Om},$$
where $\ub\Om=(\#\Om_i)_{i\in\q_0}\in\cN^{\q_0}$
and $\Om_i$ is the set of paths in \Om with endpoint $i\in Q_0$.

\begin{thr}
For any finite ideal \Om, we have
$$d_{\Om}\equiv \ub\Om_{i_0}+\ang{\ub\Om,\ub\Om}\pmod 2,$$
where $\ang{-,-}$ is the Ringel form of the quiver $Q$,
given, for $\al,\be\in\cZ^{\q_0}$, by 
$$\ang{\al,\be}=\sum_{i\in\q_0}\al_i\be_i-\sum_{a:i\arr j}\al_i\be_j.$$
\end{thr}
\begin{proof}
Let $(M,m)$ be an $i_0$-cyclic $A$-module corresponding to
\Om. Consider an exact sequence 
$$0\arr I\arr P_{i_0}\arr M\arr 0.$$
Note that all the modules in this sequence are \La-graded.
One can show, in the same way as for $\Quot$-schemes 
(see e.g.\ \cite[Prop.\ 2.2.7]{HL1}) that
the tangent space $T_{(M,m)}X$ is isomorphic
to $\Hom(I,M)$. Consider the exact sequence
$$0\arr\Hom(M,M)\arr\Hom(P_{i_0},M)\arr\Hom(I,M)\arr\Ext^1(M,M)\arr0.$$
The second map is actually zero, because any
morphism $f:P_{i_0}\arr M$ maps $I$ to zero by degree reasons.
It follows that 
$$d_\Om=\dim\Hom(I,M)=\dim\Ext^1(M,M).$$

In order to find $\dim\Ext^1(M,M)$, we will study
the Euler characteristic $\hi(M,M)$. Usually it is
defined as an alternating sum of dimensions of
$\Ext$-groups, but this gives a rather unsatisfactory
output in the case of Calabi-Yau algebras. Therefore,
for any finite dimensional \La-graded $A$-modules $L,N$, we define
$$\hi(L,N)=\sum_{k\ge0}(-1)^k\Ext^k(L,N),$$
considered as an element of the Grothendieck group
$K_\La$ of \La-graded vector spaces. 
It follows from the results of the previous section that
for any such modules $L,N$ we have
$$\Ext^k(L,N)=\Ext^{3-k}(N,L[-\ub\om])\dual=\Ext^{3-k}(N,L)\dual[\ub\om]$$
in $K_\La$.
For any $\la\in\La$, we define $t^\la\in K_\La$ to be the 
element corresponding to the one-dimensional
vector space concentrated in degree \la. 
We define a homomorphism $D:K_\La\arr K_\La$ by
the formula $t^\la\mto t^{-\la-\ub\om}$.
Then 
$$\Ext^k(L,N)=D\Ext^{3-k}(N,L).$$
It follows
from Proposition \ref{prp:resolution} that for any $i,j\in\q_0$ we have
$$\hi(S_i,S_j)=\de_{i,j}-\sum_{a:i\arr j}t^{-a}+
\sum_{b:j\arr i}t^{b-\ub\om}-\de_{i,j}t^{-\om}.$$
This implies for the module $M$:
\begin{align*}
\hi(M,M)
=\sum_{\stackrel{u,v\in\Om}{t(u)=t(v)}}t^{\wt(v)-\wt(u)}
-\sum_{\stackrel{u,v\in\Om}{a:t(u)\arr t(v)}}t^{\wt(v)-\wt(u)-a}\\
+\sum_{\stackrel{u,v\in\Om}{b:t(u)\arr t(v)}}t^{\wt(v)-\wt(u)+b-\ub\om}
-\sum_{\stackrel{u,v\in\Om}{t(u)=t(v)}}t^{\wt(v)-\wt(u)-\ub\om}.
\end{align*}
Therefore, for 
$$A:=\sum_{\stackrel{u,v\in\Om}{t(u)=t(v)}}t^{\wt(v)-\wt(u)}
-\sum_{\stackrel{u,v\in\Om}{a:t(u)\arr t(v)}}t^{\wt(v)-\wt(u)-a}$$
we have $\hi(M,M)=A-DA$.
Let us define 
$$B:=\Hom(M,M)-\Ext^1(M,M).$$ 
Then we have $\hi(M,M)=B-DB$. Therefore
$$A-B=D(A-B).$$
If $t^\la$ has coefficient $n$ in $A-B$
then also $t^{-\la-\ub\om}$ has coefficient $n$ in $A-B$.
Note that $\la\ne {-\la-\ub\om}$ for any $\la\in\La$.
Indeed, if $\ub\om=-2\la$ then it follows from the exact sequence
(see Section \ref{sec:brane})
$$0\arr \cZ\arr^{1\mto\ub\om}\La\arr\coker d_2\arr 0$$
that $\coker d_2$ has torsion, contradicting Lemma \ref{lmm:free group}.
This implies that the dimension of $A-B$ is even.
Now we compute
$$\dim A=\sum_{i\in\q_0}\ub\Om_i^2-\sum_{a:i\arr j}\ub\Om_i\ub\Om_j=\ang{\ub\Om,\ub\Om},$$
$$\dim\Hom(M,M)=\dim\Hom(P_{i_0},M)=\ub\Om_{i_0},$$
and
\enlargethispage{\baselineskip}
$$\dim\Ext^1(M,M)\equiv \dim\Hom(M,M)-\dim A
\equiv\ub\Om_{i_0}+\ang{\ub\Om,\ub\Om}\pmod 2.$$
\end{proof}

\begin{rmr}
Our result immediately implies \cite[Theorem A.3]{Young1} on the signs
in the cases $\cC^3/\cZ_n$ and $\cC^3/(\cZ_2\xx\cZ_2)$
and \cite[Theorem 2.7.1]{Szen1} on the signs in the case of conifold.
\end{rmr}

\begin{crl}
Let $Z^{i_0}(A)=\sum_{\al\in\cN^I}Z^{i_0}(A)_\al x^\al$. Then
$$Z_{DT}^{i_0}(A)=\sum_{\al\in\cN^I}(-1)^{\al_{i_0}+\ang{\al,\al}} Z^{i_0}(A)_\al x^\al.$$
\end{crl}
\section{Rationality conjecture}\label{sec:conj}
In the last section we are going to discuss
the structure of the generating function 
$Z^i(A)\in\cQ\pser{x_j|j\in Q_0}$, where 
$A=\cC Q/(\dd W)$ and $i\in Q_0$.

Let $R=\cQ\pser{x_1,\dots,x_r}$ and $R^+$ be its maximal ideal.
We endow $R$ with the structure of a \la-ring (see \cite{Getz1}
or \cite[Appendix]{M2})
by defining the Adams operations $\psi_n:R\arr R$, $n\ge1$
$$\psi_n(f(x_1,\dots,x_r)):=f(x_1^n,\dots,x_r^n).$$
We define a plethystic exponent $\Exp:R^+\arr 1+R^+$ by the formula
$$\Exp(f)=\exp\bigg(\sum_{n\ge1}\frac1n \psi_n(f)\bigg).$$
Its inverse, plethystic logarithm $\Log:1+R^+\arr R^+$, is given by
$$\Log(f)=\sum_{n\ge1}\frac{\mu(n)}n\psi_n(\log(f)).$$

For example, the MacMahon function
$$M(x,z)=\prod_{n\ge1}\left(\frac1{1-xz^n}\right)^n
=\prod_{n\ge1}\Exp(xz^n)^n=\Exp\bigg(\sum_{n\ge1}nxz^n\bigg)=
\Exp\left(\frac{xz}{(1-z)^2}\right).$$

In \cite{Young1} it was proved that the generating function 
$Z^i(A)$ in the cases of the orbifolds $\cC^3/\cZ_n$ (with a group action
$\frac1n(1,0,-1)$) and $\cC^3/(\cZ_2\xx\cZ_2)$
(with a group action $\frac12(1,0,1)\xx\frac12(0,1,1)$)
and in the case of the conifold can be represented as a product
of MacMahon functions. This allows us to formulate

\begin{conj}
Let $G$ be a consistent brane tiling, $(Q,W)$ be the corresponding
quiver potential, $A=\cC Q/(\dd W)$ be the quiver potential algebra
and $i\in Q_0$ be some vertex. Let $r=\# Q_0$. Then the power series
$$\left.\Log(Z^i(A))\right|_{x_1=\dots=x_r=x}\in\cQ\pser x$$
is a rational function.
\end{conj}

\begin{rmr}
It is not true in general that the function $\Log(Z^i(A))$
is a rational function. This can be seen already in the case
of the orbifold $\cC^3/\cZ_3$ with the group action $\frac13(1,1,1)$.
We thank Tom Bridgeland for this remark. 
\end{rmr}

We have tested this conjecture using a computer in the case of a
suspended pinch point (see e.g. \cite{Kenn1}), which is given by the quiver
\vspace{10pt}
$$\xymatrix{
&1\ar@{<->}[dr]\ar@{<->}[dl]\ar@(ul,ur)[]&\\
3\ar@{<->}[rr]&&2}$$
with potential
$$W=x_{21}x_{12}x_{23}x_{32}-x_{32}x_{23}x_{31}x_{13}+
x_{13}x_{31}x_{11}-x_{12}x_{21}x_{11},$$
where $x_{ij}$ is an arrow between the points $i$ and $j$.
Then
\begin{multline*}
\left.\Log(Z^1(A))\right|_{x_1=x_2=x_3=x}\\
=\frac{x^{11}+2x^{10}+3x^9+2x^8+5x^7+6x^6+5x^5+2x^4+3x^3+2x^2+x}{(1-x^6)^2},
\end{multline*}
\begin{multline*}
\left.\Log(Z^2(A))\right|_{x_1=x_2=x_3=x}\\
=\frac{x^{11}+x^{10}+3x^9+3x^8+5x^7+6x^6+5x^5+3x^4+3x^3+x^2+x}{(1-x^6)^2}.
\end{multline*}

We have also tested Model I of $dP3$ (see e.g.\ \cite{Kenn1}), 
which is given by the quiver
$$\xymatrix{
&1\ar[dr]\ar[ddr]&\\
6\ar[ur]\ar[rr]&&2\ar[d]\ar[ldd]\\
5\ar[uur]\ar[u]&&3\ar[ld]\ar[ll]\\
&4\ar[lu]\ar[luu]}$$
with potential
\begin{align*}
W=x_{12}x_{23}x_{34}x_{45}x_{56}x_{61}
+x_{13}x_{35}x_{51}+x_{24}x_{46}x_{62}\\
-x_{23}x_{35}x_{56}x_{62}
-x_{13}x_{34}x_{46}x_{61}-x_{12}x_{24}x_{45}x_{51}.
\end{align*}
In this case we have
\begin{multline*}
\left.\Log(Z^1(A))\right|_{x_1=\dots=x_6=x}\\
=\frac{x^{11}+x^{10}+2x^9+2x^8+5x^7+6x^6+5x^5+2x^4+2x^3+x^2+x}{(1-x^6)^2}.
\end{multline*}

Of course it would be nice to be able to write down the rational function
predicted by the conjecture just from the brane tiling data. At the 
moment we do not have such a construction.

\bibliography{../tex/fullbib}
\bibliographystyle{../tex/hamsplain}
\end{document}